\documentclass[oneside,english]{amsart}
\usepackage[T1]{fontenc}
\usepackage[latin9]{inputenc}
\usepackage{geometry}
\usepackage{babel}
\usepackage{amsbsy}
\usepackage{amstext}
\usepackage{amsthm}
\usepackage{amssymb}
\usepackage{graphicx}
\usepackage{verbatim}
\usepackage[all]{xy}
\usepackage[unicode=true,pdfusetitle,
 bookmarks=true,bookmarksnumbered=false,bookmarksopen=false,
 breaklinks=false,pdfborder={0 0 1},backref=false,colorlinks=true]
 {hyperref}
\usepackage{xcolor}
\usepackage{soul}
\makeatletter
\numberwithin{equation}{section}
\numberwithin{figure}{section}
\theoremstyle{plain}
\newtheorem{thm}{\protect\theoremname}[section]
\theoremstyle{plain}
\newtheorem{cor}[thm]{\protect\corollaryname}
\theoremstyle{plain}
\newtheorem{lem}[thm]{Lemma}

\theoremstyle{definition}
\newtheorem{defn}[thm]{\protect\definitionname}
\theoremstyle{definition}
\newtheorem{exmp}[thm]{Example}
\theoremstyle{definition}
\newtheorem{qn}{Question}
\newtheorem{rem}[thm]{\protect\remarkname}

\title{Stabilization and satellite construction of doubly slice links}
\author{Hongtaek Jung}
\address{Center for Geometry and Physics, Institute for Basic Science}
\email{htjung@ibs.re.kr}
\author{Sungkyung Kang}
\address{Center for Geometry and Physics, Institute for Basic Science}
\email{sungkyung38@icloud.com}
\author{Seungwon Kim}
\address{Center for Geometry and Physics, Institute for Basic Science}
\email{math751@ibs.re.kr}

\makeatother

\providecommand{\corollaryname}{Corollary}
\providecommand{\definitionname}{Definition}
\providecommand{\remarkname}{Remark}
\providecommand{\theoremname}{Theorem}
\renewcommand{\setminus}{-}
\begin{document}

\begin{abstract}
    A 2-component oriented link in $S^3$ is called weakly doubly slice if it is a cross-section of an unknotted sphere in $S^4$, and strongly doubly slice if it is a cross-section of a 2-component trivial spherical link in $S^4$. We give the first example of 2-component boundary links which are weakly doubly slice but not strongly doubly slice. We also introduce a new invariant $g_{st}$ of homotopically trivial links that measures the failure of a link from being strongly doubly slice and that bounds the doubly slice genus $g_{ds}$ from below. Our examples have arbitrarily large doubly slice genus but satisfy $g_{st}=1$. We also prove that the Conway-Orson signature lower bound on $g_{ds}$ is actually a lower bound on $g_{st}$.
\end{abstract}
\maketitle

\tableofcontents

\section{Introduction}

In this paper we study interference between 4-dimensional surface-link theory and 3-dimensional link theory.

 One motivation to study the interaction between 3- and 4-dimensional link theory can be found in the 4-dimensional unknotting conjecture, which states that a given embedded 2-sphere $F$ in $S^4$ bounds a 4-ball if and only if the fundamental group of the complement of $F$ is infinite cyclic. In the topological category, this was solved by Freedman \cite{freedman1984,freedman1990}, but in the smooth category, this still remains open. One way to handle knotted surfaces in 4-manifolds is to look at its 3-dimensional sections. For example, in \cite{kawauchi1982}, Kawauchi, Shibuya and Suzuki introduced the normal form, which puts a knotted surface in a nice Morse position so that we can represent the surface by its 3-dimensional cross-section (link) and bands. In this presentation, if one can get a trivial link or trivial knot in its 3-dimensional cross-section, then the surface looks very close to the unknotted surface. For a 2-knot, it is conjectured that a surface-knot is smoothly unknotted if and only if its 3-dimensional cross-section is an unknot. Then we can ask a naive but natural question: if a 3-dimensional cross-section of a $2$-knot is knotted, is the $2$-knot knotted? The answer is negative, since the connected sum of any knot with its mirror is always a 3-dimensional cross-section of a 1-twist spun knot, which is isotopic to a trivial $2$-knot. However, from this observation, the following interesting question can occur: which knots can be a 3-dimensional cross-section of a trivial $2$-knot?

The simplest question in this direction is to ask what kind of properly embedded surfaces can a given knot $K$ bound in $B^4$. When $K$ bounds a disk -- locally flat or smooth, depending on your choice of a category -- we say that $K$ is slice. The invariant that measures the failure of a knot $K$ from being slice is called the 4-genus $g_4(K)$, which is defined by the minimum genus of properly embedded surfaces in $B^4$ that bound $K$. 

Now consider gluing two $B^4$ along their boundaries to obtain $S^4$. In this point of view, we see that a knot $K$ is slice if and only if it is a cross-section of some embedded $S^2$ in $S^4$. But this embedded 2-sphere $F$ might be knotted in $S^4$; when $F$ can be chosen to be unknotted in $S^4$, we say that $K$ is \emph{doubly slice}. The failure of a knot $K$ from being doubly slice is measured by the \emph{doubly slice genus} $g_{ds}(K)$, that is the minimum genus of unknotted embedded surfaces in $S^4$ whose cross-section is $K$. This invariant was first introduced in \cite{livingston2015}. Note that a slice knot, and even a ribbon knot, can have arbitrarily large doubly slice genus, as shown first in \cite{chen2021}.

For a link $L$ in $S^3$, we need more elaborate definition for the doubly slice genus. Given an orientation on $L$ and a coloring $\mu$ on $L$, i.e. a surjective map $\mu:\pi_0(L) \rightarrow C$ for some set $C$, we say that an $n$-component \emph{oriented} link $L$ is $\mu$-\emph{doubly slice} ($1\le \mu \le n$) if there exists a $\vert C \vert$-component trivial spherical link $F=\cup_{c\in C} F_c$ in $S^4$ such that $F_c \cap S^3 =\cup_{\mu(L_0 )=c} L_0$. For any given oriented link, we use the term weakly (strongly) doubly slice to denote the doubly sliceness when all components have different (same) color; note that being strongly doubly slice does not depend on the orientations. Introduced by Conway-Orson  \cite{conway2021},  the (strong) doubly slice genus $g_{ds}(L)$ of an unoriented $n$-component link $L$
\[
g_{ds}(L) = \min\{g(F)\,|\, F\text{ is an } n\text{-component trivial surface-link in }S^4\text{ such that }L=F\cap S^3\},
\]
measures how far $L$ is from being strongly doubly slice.

A strongly doubly slice $n$-component link is the cross-section of the boundary of $n$ disjoint embedded 3-balls. Thus, it is natural to ask a more general question: When a given $n$-component link $L\subset S^3$ can be realized as the cross-section of the boundary of $n$ disjoint 3-manifolds. An $n$-component link $L$ with this property is called a \emph{boundary link}. Obviously, strongly doubly slice links are boundary links and their components are doubly slice; also, they are weakly doubly slice with respect to all orientations. However, while Conway-Orson \cite{conway2021} constructed several obstructions which can detect such an example using abelian invariants of links, but no examples of weakly doubly slice boundary links which are not strongly doubly slice were previously known. Also, McCoy-McDonald \cite{mccoy2021} proved that some 2-component pretzel links are weakly doubly slice for all orientations but not strongly doubly slice, but we will see in Section \ref{motivatingex} that some of them are not boundary links; we conjecture that all such pretzel links are non-boundary.

For knots, doubly slice genus is always well-defined, as every knot in $S^3$ is a cross-section of an unknotted closed surface in $S^4$. However, not all links can be written as cross-sections of trivial surface-links in $S^4$; a link appears as such a cross-section if and only if it is a boundary link. Thus, it is natural to introduce a modification of $g_{ds}$ so that it can be defined for more general class of links. To do so, we start from the observation that every surface-knot in $S^4$ can be unknotted by a sequence of stabilizations followed by  destabilizations. Under a mild condition on a surface-link, one can always find an unlinking sequence, where we do not allow stabilizations to connect different components, that brings the given surface-link to a trivial surface-link. Measuring complexity of unknotting sequences yields a new invariant $g_{st}$ called the stabilization genus, whose precise definition is given in Section \ref{secdsg}. One can interpret $g_{st}$ as a distance between a given surface-link and a trivial surface-link. Associated to $g_{st}$ is its 3-dimensional version, also denoted as $g_{st}$, which is defined as 
\[
g_{st}(L)=\min\{ g_{st}(F)\,|\, F\text{ is a }n\text{-component surface-link such that } F\cap S^3 = L\}.
\]
The stabilization genus $g_{st}$ also measures how far a given link $L$ is from being strongly doubly slice like the doubly slice genus $g_{ds}$ did. Furthermore, we will see in Section \ref{secdsg} that $g_{st}$ is defined for  homotopically trivial links, which is a less restrictive condition than being a boundary link.

In this paper, we give the first examples of $2$-component boundary links which are weakly doubly slice with respect to all orientations but not strongly doubly slice. In particular, we prove the following theorem.
\begin{thm}
\label{mainthm1}
For any positive integer $n$, there exists a weakly doubly slice boundary 2-component link $L_n$ consisting of doubly slice components such that $g_{st}(L_n)=1$ and $g_{ds}(L_n) \ge 2n$.
\end{thm} 

Recall that Conway-Orson \cite{conway2021} gave a lower bound on the doubly slice genus of $\mu$-colored links using multivariable signature, i.e. $\vert \sigma_L(\omega)\vert \le g_{ds}(L)$ for any $\omega\in (S^1\setminus\{1\})^\mu$. It turns out that this lower bound vanishes identically for our examples. In general, we will see that $\vert \sigma_L\vert$ actually gives a lower bound on a smaller quantity $g_{st}$. 

\begin{thm}
Let $L$ be a  $\mu$-colored link in $S^3$. Then $\vert \sigma_L(\omega)\vert \le g_{st}(L)$ for any $\omega\in (S^1\setminus\{1\})^\mu$.
\end{thm}

We can extend our discussion to higher dimensional cases. Recall that an $n$-component $m$-link is an oriented embedded $n$-copies of $S^m$ in $S^{m+2}$. An $m$-link is unknotted if each of its component is isotopic to the standard $m$-sphere. An $n$-component $m$-link is boundary if it admits $n$ disjoint Seifert solids. It is called weakly (strongly, resp.) doubly slice is it can be realized as a cross-section of the trivial $(m+1)$-knot (trivial $n$-component $(m+1)$-link, resp.). Although less quantitative, we have a similar result for higher dimensional cases.

\begin{thm}\label{mainthm4}
Let $m$ be any natural number not equal to 2.  There is a 2-component boundary $m$-link which is unknotted and weakly doubly slice but not strongly doubly slice. 
\end{thm}

Our examples are constructed by performing satellite operations, where we use companion knots which are slice but not doubly slice and patterns which are induced by Brunnian links. When the companion is topologically or smoothly doubly slice, the satellite knot that we obtain is also topologically or smoothly doubly slice, respectively. This observation leads us to the following theorem, which will be proven in Section \ref{Topologicallynotsmoothly}.

\begin{thm}
\label{mainthm2}
There exists a 2-component boundary link which is smoothly weakly doubly slice for all orientations and topologically strongly doubly slice but not smoothly strongly doubly slice.
\end{thm}

\subsection*{Acknowledgements:} This work was supported by Institute for Basic Science (IBS-R003-D1).

\section{Motivating Example}\label{motivatingex}
In \cite{mccoy2021}, it is shown that for 2-component 4-strand pretzel links, the notion of slice and weakly doubly slice coincide. In particular, such a link $L$ is weakly doubly slice if and only if it is of the form $L=P(a,b,-b,-a)$, where at most one of $a$ and $b$ is even. When $a=\pm b$, then it is known that $L$ is strongly doubly slice, and if $a\ne b$ and $a,b$ are not relatively prime, then it is known that $L$ is not strongly doubly slice. Thus it is natural to ask the following question.

\begin{qn}
Is there a 2-component pretzel link which is a weakly doubly slice boundary link but not strongly doubly slice?
\end{qn} 

We give here an empirical evidence that there might not exist such a link.
\begin{thm}
\label{claytonboundary}
For any $p,n\ge 1$ such that $n$ is not a multiple of $2(2p+1)$, none of the 4-strand pretzel links $L_{p,n}=P(2p+1,2n,-2n,-2p-1)$ are boundary links.
\end{thm}
\begin{proof}
$L_{p,n}$ has one unknotted component, which we will denote by $U$. Denote the Seifert longitude of $U$ by $\lambda$, which we will regard as an element of $\pi_1(S^3\setminus L_{p,n})$. It is known that if a given link is a boundary link, then the Seifert longitudes of its components are contained in the second commutator of its link group, and so their Fox derivatives are trivial \cite{crowell1971}. Using this criterion, we will show, by explicit calculation, that the image of $\lambda$ under the Fox differential $\partial :\pi_1 (S^3\setminus L_{p,n} )\rightarrow A_{L_{p,n}}$ is nontrivial, where $A_{L_{p,n}}$ denotes the Alexander module of $L_{p,n}$.

Using Wirtinger presentation, we see that $\pi_1(S^3\setminus L_{p,n})$ has four generators, $a$, $b$, $c$, and $d$, as shown in Figure \ref{pretzel}, and three relators $r_1$, $r_2$, and $r_3$, defined as below:
\begin{align*}
    r_1 &= (ab^{-1})^p aba^{-1} (ba^{-1})^p (ad^{-1})^p a d^{-1} a^{-1} (da^{-1})^p,\\
    r_2 &= (ab^{-1})^p aba^{-1} (ba^{-1})^p (bc^{-1})^n b (cb^{-1})^n,\\
    r_3 &= \text{some word consisting only of }a,c\text{ and }d.
\end{align*}
Write the meridians of $L_{p,n}$ in $H_1 (L_{p,n};\mathbb{Z})$ as $s,t$, where $t$ denotes the meridian of $U$. Then the images of $a,b,c,d$ under the abelianization map are given by $s^{-1}$, $s$, $t$, and $s$, respectively. The derivatives of the relators $r_i$ with respect to the generator $b$ is given as follows:
\begin{align*}
    \partial_{b}(r_1) &=s^{-2p-1}(1-s+s^2-\cdots-s^{2p-1}+s^{2p}),\\
    \partial_{b}(r_2) &=s^{-2p-1}(1-s+s^2-\cdots-s^{2p-1}+s^{2p})+(t^{-1}-1)(1+st^{-1}+\cdots+(st^{-1})^{n-1}),\\
    \partial_{b}(r_3) &= 0.
\end{align*}
Suppose that $L_{p,n}$ is a boundary link. Then for any choice of longitude $\lambda$ of any component of $L_{p,n}$, we should have $\partial_{b} (\lambda) = p(s,t)\partial_b (r_1)+q(s,t)\partial_b(r_2)$ for some $p,q\in \mathbb{Z}[s^{\pm 1},t^{\pm 1}]$. Specializing this equation at $t=1$ gives:
\begin{align}
    \partial_b(\lambda)\vert_{t=1} &= p(s,1)\partial_b(r_1)\vert_{t=1} + q(s,1)\partial_b(r_2)\vert_{t=1} \\
    &= s^{-2p-1}(1-s+s^2-\cdots-s^{2p-1}+s^{2p})(p(s,1)+q(s,1)).
\end{align}

Now observe that specializing $\partial_b(\lambda)$ at $t=1$ is equivalent to ignoring $c$ while computing the Fox derivative, which is equivalent to ignoring the unknotted component in $L_{p,n}$. So any longitude $\lambda$ of the unknotted component of $L_{p,n}$ becomes $b^n d^{-n}$ after ignoring $c$, and thus we have 
\[
\partial_b(\lambda)\vert_{t=1}=\partial_b(b^n d^{-n}) = 1+s+\cdots+s^{n-1}.
\]
Hence, the polynomial $1+s+\cdots+s^{n-1}$ should be divisible by $1-s+s^2-\cdots-s^{2p-1}+s^{2p}$. Since $e^{\frac{\pi i}{2p+1}}$ is a root of $1-s+s^2-\cdots-s^{2p-1}+s^{2p}$, it should also be a root of $1+s+\cdots+s^{n-1}$. In particular, we should have 
\[
e^{\frac{n\pi i}{2p+1}}=1,
\]
which would imply that $n$ is a multiple of $2(2p+1)$, a contradiction. Therefore $L_{p,n}$ is not a boundary link.
\end{proof}

\begin{figure}[hbt]
    \centering
    \includegraphics[width=0.6\textwidth]{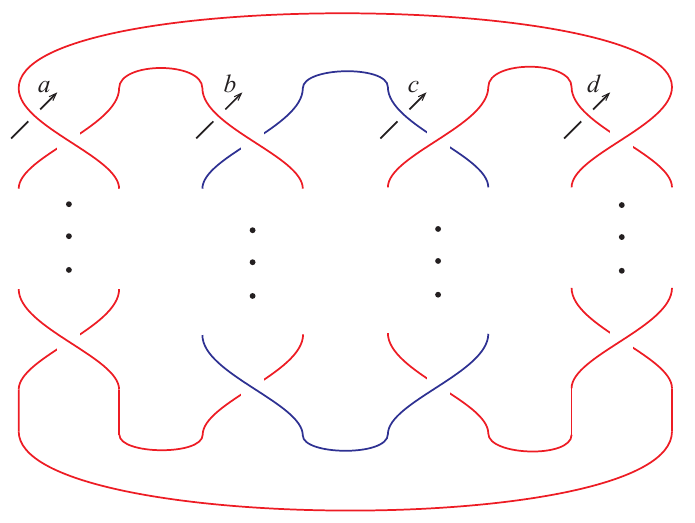}
    \caption{The pretzel link $L_{p,n}=P(2p+1,2n,-2n,-2p-1)$. Its two components are colored in red and blue.}
\label{pretzel}
\end{figure}

\section{Satellites along 3-component Brunnian links}
In this section, we will describe a construction which takes a 3-component Brunnian link with a distinguished component and a slice knot as an input and gives a weakly doubly slice boundary link as an output. The links that we get by this construction are potentially not strongly doubly slice; we will see that some of them are actually not strongly doubly slice in later sections.

Throughout this paper, all knots and links are assumed to be oriented. 

\begin{defn}[See also \cite{kim2020}]
\label{satellitedefn}
Let $P$ be a spherical link and $C$ be a spherical $n$-knot embedded in $S^{n+2}$ and a $(n+2)$-manifold $X$ respectively. Assume $C$ has a product neighborhood in $X$. Consider a simple loop $\gamma \subset S^{n+2}-\nu(P)$. Then there exists a diffeomorphism $\rho : \overline{S^{n+2} - \nu(\gamma)} \rightarrow \nu(C)$, where $\nu(C)$ denotes a tubular neighborhood of $C$. Let $K = \rho(P) \subset X$. We call $K$ the \emph{satellite link in $X$ of companion $C$ with pattern $(P,\gamma)$}. Equivalently, 
$$(X, K) = ((\overline{X - \nu(C))}\bigcup_{\partial \rho}(\overline{S^{n+2} - \nu(\gamma)}), P),$$ where
$$\partial \rho = \rho\restriction_{\partial(\overline{S^{n+2} - \nu(\gamma)})}\; : \partial(\overline{S^{n+2} - \nu(\gamma)}) \rightarrow \partial \nu(C) \simeq \partial (\overline{X - \nu(C)})$$
 and 
$$P \subset S^{n+2} - \nu(\gamma) \subset (\overline{X - \nu(C))}\bigcup_{\partial \rho}(\overline{S^{n+2} - \nu(\gamma)}) \simeq X.$$
 We say an oriented satellite $n$-link is degree $0$ if $[P] = 0 \in H_{n}(S^{n+2} - \nu(\gamma)) \simeq \mathbb{Z}$.
\end{defn}

Using the above notion of satellites along knotted sphere companions, we can now prove the following theorem.

\begin{thm}
\label{satthm}
Let $K$ be a doubly slice (resp. slice) knot and $L$ be an oriented link which is strongly (resp. weakly) doubly slice. Then $0$-framed satellite link which is obtained by taking $L$ as a pattern (resp. degree $0$ pattern) and $K$ as a companion is a  strongly (resp. weakly) doubly slice link.
\end{thm}

\begin{proof}
Suppose first that $K$ is doubly slice and $L$ is strongly doubly slice. Let $S_L$ be a 2-component trivial 2-link in $S^4$ such that $S_L\cap S^3=L$, where $S^3 \subset S^4$ is a standard embedding of $S^3$. Also, let $S_K$ be a trivial $2$-knot in $S^4$ such that $S_K \cap S^3 = K$. If we specify $\gamma \subset S^3 - L \subset S^4 - S_L$, then we can think of a $0$-framed satellite link $L(K)$ and a satellite $2$-link $S_L(S_K)$ such that $S_L(S_K) \cap S^3 = L(K)$. 

For a given pattern, the isotopy classes of its satellite only depends on the isotopy classes of the companion. Since $S_L$ itself can be considered as a satellite $2$-link $S_L(U)$, where $U$ is a trivial $2$-knot, and $S_K$ is isotopic to $U$, so $S_L(S_K)$ is isotopic to $S_L$, which is trivial. Therefore, $L(K)$ is a cross-section of  a trivial $2$-link, so $L(K)$ is strongly doubly slice.

Suppose that $K$ is slice and $L$ is weakly doubly slice. Let $S_L$ be a trivial $2$-knot in $S^4$ such that $S_L\cap S^3=L$. Also, let $S_K$ be a $2$-knot in $S^4$ such that $S_K \cap S^3 = K$. Note that $S_K$ does not need to be trivial, but it exists since we can double a slice disk bounded by $K$. Let $\gamma$, $L(K)$, $S_L(S_K)$ be the same objects in the previous proof. Since we assume that $L$ as a degree 0 pattern, $\gamma$ can be isotoped in $S^4 - S_L$ so that it does not intersect the ball bounded by $S_L$. In other words, $S_L$ bounds a $3$-ball in $S^4 - \gamma$. Hence, if we do the satellite operation, $S_L(S_K)$ still bounds a ball in $\nu(S_K) \subset S^4$, hence it is unknotted. Therefore, $L(K)$ is weakly doubly slice. 
\end{proof}

Recall that a link $L$ is said to be \emph{Brunnian} if any proper sublink of $L$ is an unlink. Given a 3-component Brunnian link $L$ together with a distinguished component $U\subset L$, we can consider $L- U \subset S^3 - U$ as a pattern link $P_{L,U}$ in the solid torus $D^2 \times S^1$, where we give a framing on the solid torus by the Seifert framing on $U$. Then, given any knot $K\subset S^3$, we can consider its $0$-framed satellite $P_{L,U}(K)$, which is a 2-component link in $S^3$.

\begin{lem}
$P_{L,U}(K)$ is always a boundary link, and every component of $P_{L,U}(K)$ is unknotted. Also, if $K$ is slice, then $P_{L,U}(K)$ is weakly doubly slice for both of its quasi-orientations, and if $K$ is doubly slice, then $P_{L,U}(K)$ is strongly doubly slice.
\end{lem}
\begin{proof}
Denote the components of $P_{L,U}\subset D^2 \times S^1$ by $A,B$. Since $L$ is Brunnian and $L\cup A=L- B$ is its proper sublink, $L\cup A$ should be an unlink, which means that $A\subset D^2 \times S^1$ is a local unknot, i.e. bounds an embedded disk. Hence the satellite $A(K)$ is unknotted, and by symmetry, $B(K)$ is also unknotted. since $P_{L,U}(K)=A(K)\cup B(K)$, we deduce that $P_{L,U}(K)$ has unknotted components.

To see that $P_{L,U}(K)$ is a boundary link, we observe that $A\cup B=L- U$ is an unlink. Thus there exist disjoint embedded disks $D_1$ and $D_2$ in $S^3$, bounding $A$ and $B$, respectively, so that each of $D_i$ intersects $U$ transversely. Then the surface $\Sigma = (D_1 \cup D_2)- \overline{N(U)}$, where $N(U)$ denotes a neighborhood of $U$, has boundary $(L - U)\cup (\text{parallel copies of }U)$. Then attaching parallel copies of a Seifert surface of $K$ gives a disjoint Seifert surface of $P_{L,U}(K)$. Hence $P_{L,U}(K)$ is a boundary link. The statements about doubly sliceness of $P_{L,U}(K)$ follow directly from Theorem \ref{satthm}.
\end{proof}

\begin{exmp}
Let $L$ be the Borromean ring and $U$ be any component of $L$. Then $P_{L,U}$ is the Bing doubling pattern.
\end{exmp}

\begin{exmp}
Using 3-component Brunnian links other than the Borromean ring, we obtain more complicated patterns which can be used to produce more examples which are not Bing doubles. For example, let $L_n=R\cup G\cup B$ be the 3-component link as shown in the left of Figure \ref{newbrunnian}. Then $L_n$ is Brunnian and $P_{L_n,G}$ is the pattern shown in the right of Figure \ref{newbrunnian}.
\end{exmp}

\begin{figure}
    \centering
    \includegraphics[width=\textwidth]{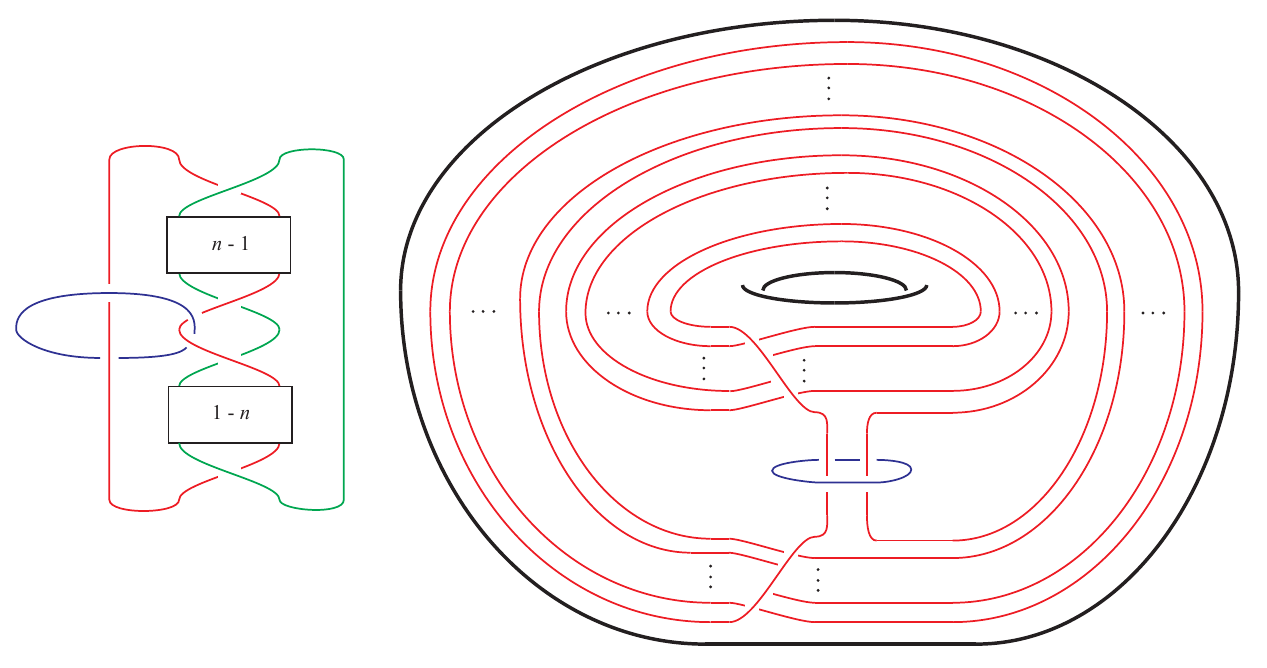}
    \caption{\textbf{Left}: The Brunnian link $L_n=R\cup G\cup B$, where the components $R$, $G$, and $B$ are drawn in red, green, and blue, respectively. Notice that $L_1$ is the Borromean ring. The twist regions are \textbf{Right}: The induced pattern $P_{L_n,G}$, which is a 2-component link in a solid torus.}
\label{newbrunnian}
\end{figure}

\section{doubly slice genus and stabilization genus}\label{secdsg}

In this section, we recall the definition of the doubly slice genus of links in $S^3$. We also introduce the stabilization genus $g_{st}$ for surface-links and for links in $S^3$ which is motivated from the fact that a large class of non-trivial surface-links can be turned into a trivial surface-link via a sequence operations called stabilizations and destabilizations. Because the set of links in $S^3$ that can be realized as the cross-sections of surface-links that admit such unlinking sequences is much larger than the set of cross-sections of trivial surface-links, the 3-dimensional counterpart, which will also be called the stabilization genus, is well-defined for a larger class of links than $g_{ds}$ is.

At the end, we prove that the absolute value of the multivariable signature is a lower bound for $g_{st}$. From now, all surface-links will be assumed to be oriented.

We begin with the well-known definition of doubly slice genus. Here, we are interested in 2-component links only. In general, however, one can define so-called the $\mu$-doubly slice genus for any $\mu$-colored $n$-component links. What we call the doubly slice genus $g_{ds}$ here is the same as the $2$-doubly slice genus $g_{ds}^2(L)$ for 2-colored, 2-component link $L$ in Conway-Orson \cite{conway2021}.

\begin{defn}
We say that a 2-component link $L$ is a \emph{cross-section} of a 2-component surface-link $\Sigma$ if $L=L_1 \cup L_2$, $\Sigma=\Sigma_1 \cup \Sigma_2$, and $L_j = \Sigma_j \cap S^3$ for each $j=1,2$. If $L$ is a cross-section of a trivial 2-component link, we define its \emph{doubly slice genus} using the formula
\[
g_{ds}(L) = \min\{ g(F_1)+g(F_2)\,\vert\, L\text{ is a cross-section of }F=F_1 \cup F_2\text{ and }F\text{ is a trivial surface-link}\},
\]
where $g(F_i)$ denotes the genus of $F_i$. If $L$ cannot be represented as such a cross-section, we set $g_{ds}(L)=\infty$.
\end{defn}

Not every links can be realized as the cross-section of a trivial surface-link. If a 2-component link $L$ were the cross-section of a trivial surface-link $F$, the intersection of a handlebody that $F$ bounds and the equatorial $S^3$ would give us a pair of disjoint Seifert surfaces of $L$. Thus, $L$ is a boundary link. The converse statement also holds:

\begin{thm}
A 2-component link $L$ has finite $g_{ds}(L)$ if and only if $L$ is a boundary link.
\end{thm}
\begin{proof}
Suppose that $g_{ds}(L) < \infty$. Then there exists a trivial 2-component surface-link $F=F_1 \cup F_2$ such that $L_i =F_i \cap S^3$ for each $i=1,2$, where $L=L_1 \cup L_2$. Denote the disjoint handlebodies bounded by $F_i$ as $H_i$. Without loss of generality, we may assume that $H_i$ transversely intersects the equatorial $S^3$. The intersection $H_i \cap S^3$ is always a union of disjoint embedded surfaces in $S^3$; among them, exactly one component bounds $L_i$, which we will denote as $\Sigma_i$. Then, by construction, $\Sigma_i$ is a Seifert surface of $L_i$ and $\Sigma_1 \cap \Sigma_2 = \emptyset$. Therefore $L$ is a boundary link.

Now assume that $L$ is a boundary link. Denote a pair of disjoint Seifert surfaces bounded by the components of $L$ as $\Sigma_1$ and $\Sigma_2$. Doubling each $\Sigma_i$ and gluing them along their boundaries give a 2-component trivial link in $S^4$ whose cross-section is $L$. Therefore $g_{ds}(L) < \infty$.
\end{proof}

Now we define the stabilization genus $g_{st}$ for links and for surface-links. This invariant $g_{st}$ is defined by complexity of unlinking procedures of a given surface-link via stablizations and destabilizations. Hence, we need to make it clear what do we mean by stabilizations and destabilizations in our context.

\begin{defn}
Let $F$ be a surface-link in a $4$-manifold $X$. Let $h = (D^2 \times I) \subset X$ be a  $3$-ball such that $h \cap F = D^2 \times \{0,1\}$ and $d = (D^2 \times I) \subset X$ be also a $3$-ball such that $d \cap F = \partial D^2 \times I$. A \emph{stabilization of $F$ along $h$} is a surface link $F_h = (F - h) \cup (\partial h - (h \cap F))$ and a \emph{destabilization of $F$ along $d$} is a surface link $F_d = (F-d) \cup (\partial d - (d \cap F))$. We also simply say $F_h$ is a stabilization of $F$ and $F_d$ is a destabilization of $F$. 
\end{defn}

We give the definition of an unlinking sequence. Note that we only allow stablizations to be done along the same component. 

\begin{defn}
Given a $p$-component surface-link $F$ in $S^4$, an \emph{unlinking sequence} of $F$ is a sequence $\{ F=F_1,F_2,\cdots,F_m,F_{m+1},\cdots,F_n \}$ of $p$-component surface-links such that the following conditions are satisfied. 
\begin{enumerate}
    \item $F_{i+1}$ is a stabilization of $F_i$ for each $i<m$;
    \item $F_{i+1}$ is a destabilization of $F_i$ for each $i\ge m$;
    \item $F_n$ is a trivial surface-link.
\end{enumerate}
\end{defn}

\begin{rem} Note that destabilization of a surface link is a reverse operation of a stabilization.
\end{rem}

Now we can define the stabilization genus for surface-links. 

\begin{defn}
Let $F$ be a surface-link. Denote the set of unlinking sequences of $F$ by $\mathcal{UL}_{F}$. Then we define
\[
g_{st}(F)= \min_{\{F_1,\cdots,F_n\}\in\mathcal{UL}_{L}} \max_{i=1,\cdots,n} \, g(F_i).
\]
If $F$ does not admit an unlinking sequence, i.e. $\mathcal{UL}_{F}=\emptyset$, then we set $g_{st}(F)=\infty$.
\end{defn}

For a trivial surface-link $F$ we have $g_{st}(F) = g(F)$; even if $F$ is not trivial surface-link, $F$ may admit an unlinking sequence. We may then ask under which condition, the given surface-link $F$ has an unlinking sequence. It turns out that every homologically unlinked surface-link always has unlinking sequences and, therefore, has finite $g_{st}$. 

\begin{defn}
A surface-link $F\subset S^4$ is \emph{homologically unlinked} if any component $F_0$ of $F$ is null-homologous in $S^4 - (F - F_0)$.
\end{defn}


\begin{thm}
Any homologically unlinked surface-link has finite $g_{st}$.
\end{thm}
\begin{proof}
Choose any component $F_1\subset F$. Since $F_1$ is null homologous in $S^4 - (F - F_1)$, $F_1$ bounds a 3-manifold $H_1$ in $S^4 - (F - F_1)$. Consider a handlebody decomposition of $H_1$ with a single $0$-handle. Then we drill out a neighborhood of a co-core of every $2$-handle of $H_1$. These operations will stabilize $F_1$ to get a surface $F^1_1$. After that, we remove a neighborhood of a co-core of every $1$-handle of $H_1$. These operations will destabilize $F^1_1$ to get a surface $F^2_1$. Since we remove every other handle except $0$-handle, $F^2_1$ bounds a ball in $S^4 \setminus (F\setminus F_1)$. Therefore, $F^2_1 \cup (F - F_1)$ is a split union of $F^2_1$ and $F - F_1$. We can then iterate the process for components of $F - F_1$ to obtain a trivial link.
\end{proof}


Now we move on to the 3-dimension and introduce the stabilization genus for $p$-component links in $S^3$. The most natural definition is the following. 

\begin{defn}
Given a $p$-component link $L$, denote the set of $p$-component surface-links with finite $g_{st}$ and having $L$ as its cross-section by $\mathcal{SL}_{L}$. We define its \emph{stabilization genus} as follows:
\[
g_{st}(L) = \min_{F\in \mathcal{SL}_{L}} g_{st}(F).
\]
If $\mathcal{SL}_{L}$ is empty, then we set $g_{st}(L)=\infty$.
\end{defn}

We showed that the doubly slice genus $g_{ds}(L)$ for a link $L$ in $S^3$ has a finite value if and only if $L$ is boundary. One nice feature of $g_{st}$ is that it has finite value not only for boundary links, but also for homotopically trivial links. Recall that a link in $S^3$ is said to be \emph{homotopically trivial} if it is link-homotopic to an unlink. It is clear that every boundary link is homtopically trivial.

\begin{lem}
\label{linkingzero}
A link $L$ satisfies $g_{st}(L)<\infty$ if and only if $L$ is homotopically trivial.
\end{lem}
\begin{proof}
Suppose that $L$ is homotopically trivial. Since $L$ is link homotopic to an unlink, tracking the crossing changes gives a cobordism $C$ from an unlink $U$ to $L$. See, for instance, \cite[Figure. 2]{audoux2014} for the realization.  Since we are not allowed to perform crossing changes between different components, the doubled cobordism $\bar{C}\circ C$ is isotopic to a stabilization of $U\times I$. Capping it of by boundary-parallel disks gives a surface-link $F$ which has $L$ as its cross-section. By construction, $F$ is isotopic to a stabilization of a trivial spherical link, so it admits an unlinking sequence. Therefore $g_{st}(L)$ is finite.
Now suppose that $L$ is not homotopically trivial. Then $L$ does not bound a strong slice surface in $B^4$, so it does not arise as a section of a surface-link in $S^4$. Therefore $g_{st}(L)=\infty$.
\end{proof}

\begin{rem}
It is obvious by definition that $L$ is strongly doubly slice if and only if $g_{ds}(L)=0$ if and only if $g_{st}(L)=0$. Furthermore, we always have $g_{st}(L) \le g_{ds}(L)$, since trivial surface-links admit (trivial) unlinking sequences. However, we will see that the lower bound of $g_{ds}$ coming from multivariable link signature is actually a lower bound of $g_{st}$. 
\end{rem}

The remaining part of this section is devoted to proving that the absolute value of the multivariable signature gives a lower bound of $g_{st}$. To this end, we recall the construction of the multivariable singnature. 

Let $X$ be a CW-complex and let $\phi=(\phi_1,\phi_2):\pi_1(X) \to \mathbb{Z}^2$ be a homomorphism. Choose an element $\omega=(\omega_1, \omega_2) \in \mathbb{T}_{\ast}^{2}:=(S^1\setminus\{1\})^2$. Then $\pi_1(X)$ acts on $\mathbb{C}$ by
\[
\gamma \cdot z = z \omega_1 ^{\phi_1(\gamma)}\omega_2 ^{\phi_2(\gamma)}
\]
giving $\mathbb{C}$ a structure of $\pi_1(X)$-module.  We use the notation $\mathbb{C}_\omega$ when we regard $\mathbb{C}$ as the $\pi_1(X)$-module with respect to the chosen $\omega$. Note that, since $\mathbb{Z}^2$ is abelian, we can view $\mathbb{C}_\omega$ as a $H_1(X)$-module.

Recall that the homology of $X$ with twisted coefficient $\mathbb{C}_\omega$, $H_{\ast}(X;\mathbb{C}_\omega)$, is defined by the homology of the chain complex
\[
\cdots \to C_2(\tilde{X}) \otimes_{\Lambda} \mathbb{C}_\omega \to C_1(\tilde{X})\otimes_{\Lambda} \mathbb{C}_\omega \to C_0(\tilde{X})\otimes_{\Lambda} \mathbb{C}_\omega \to 0
\]
where $\Lambda=H_1(X)$ and where $C_{\ast}(\tilde{X})$ is the usual (cellular) chain complex of the universal abelian cover $\tilde{X}$ of $X$ equipped with the action of $\Lambda$ by deck transformations.

Throughout this paper, we will concentrate on the case when $X$ is the exterior $X_F$ of a 2-component surface-link $F=\Sigma_1\cup \Sigma_2$ in $S^4$. Unless otherwise stated, we will use the homomorphism $\phi:\pi_1(X_F) \to \mathbb{Z}^2$ given by $\gamma \mapsto (\operatorname{lk}(\gamma, \Sigma_1), \operatorname{lk}(\gamma, \Sigma_2))$ to define $\mathbb{C}_\omega$.  

The proof of the following lemma is inspired by \cite{cochran2003} and \cite{conway2021}. 
\begin{lem}
\label{lem1}
Let $F_0$ be a 2-component trivial surface-link and $F$ be a surface-link obtained from $F$ by stabilizing its components. Let $X_F$ be the exterior of $F$. Then for any $\omega\in \mathbb{T}_{\ast}^{2}$, we have $\mathbf{dim}_{\mathbb{C}}\, H_1 (X_F,\mathbb{C}_{\omega}) \le 1$.
\end{lem}
\begin{proof}
Since $F$ is obtained by stabilization of a 2-component trivial surface-link $F_0$, we see that $\pi_1(X_F)\simeq \pi_1(X_{F_0})/K\simeq \mathbb{F}_2/K$ for some normal subgroup $K$ contained in the commutator subgroup.  Consider the space $Y:=S^1 \vee S^1$ and a continuous map $f:Y\to X_{F}$ that sends each circle to a meridian loop that normally generates $\pi_1(X_F)$. This yields the induced map $f_{\ast}:H_1(Y;\mathbb{C}_\omega)\to H_1(X_F ;\mathbb{C}_\omega)$. Clearly  $f_{\ast}$ is surjective. Hence it is enough to show that $\mathbf{dim}_{\mathbb{C}} H_1(Y;\mathbb{C}_\omega) \le 1$. In fact, we will show that $\mathbf{dim}_{\mathbb{C}} H_1(Y;\mathbb{C}_\omega) = 1$.  

By definition, $H_1(Y;\mathbb{C}_\omega)$ can be computed from the chain complex
\[
0\to C_1(\tilde{Y})\otimes_{\mathbb{Z}^2}\mathbb{C}_\omega \overset{\partial_1}{\to} C_0(\tilde{Y})\otimes_{\mathbb{Z}^2} \mathbb{C}_\omega\to 0. 
\]
Lift the 0-cell $p$ and 1-cells $x_1, x_2$ of $Y$ to the 0-cell $\tilde{p}$ and 1-cells $\tilde{x}_1, \tilde{x}_2$ of the universal abelian cover $\tilde{Y}$. Then we know that  $C_1(\tilde{Y})\otimes_{\mathbb{Z}^2}\mathbb{C}_\omega = \mathbb{C}\tilde{x}_1 \oplus \mathbb{C} \tilde{x}_2$, and that $C_0(\tilde{Y})\otimes_{\mathbb{Z}^2} \mathbb{C}_\omega = \mathbb{C} \tilde{p}$. Under this identification, the differential $\partial_1$ can be written as $z \tilde{x}_i \mapsto (f(x_i)\cdot z-1)\tilde{p}$. Since  $(\operatorname{lk}(\gamma, \Sigma_1), \operatorname{lk}(\gamma, \Sigma_2))\ne (0,0)$ for any nontrivial $\gamma\in \pi_1(X_F)$, we see that $\partial_1$ is surjective. This proves that $\mathbf{dim}_{\mathbb{C}}\mathbf{ker}(\partial_1) = \mathbf{dim}_{\mathbb{C}} H_1(Y;\mathbb{C}_\omega) =1$. 
\end{proof}

The following theorem and its proof is a slight generalization of \cite{conway2021}. For the sake of completness, we include the full proof. Note that, since $g_{st}(L) \le g_{ds}(L)$, we recover \cite[Theorem 3.4]{conway2021} for $\mu=2$ case. 

\begin{thm}
Let $L$ be a homologically unlinked 2-component link and $\sigma_L$ be its multivariable signature. Then $\vert \sigma_L(\omega) \vert \le g_{st}(L)$ for any $\omega\in \mathbb{T}_{\ast} ^{2}$.
\end{thm}

\begin{proof}
By definition, there exist surface-links $F_0$ and $F$ such that $F$ is common stabilization of $F_0$ and some trivial 2-component link, and $g(F)=g_{st}(L)$. Stabilizations are performed on $F_0$ along arcs whose endpoints lie on the same component of $F_0$; by perturbing the arcs so that they intersect transversely with the equatorial $S^3$, we see that there exists a 2-coloring on some unlink $U$ so that the 2-colored link $L\sqcup U$ is a cross-section of $F$. For simplicity, we write $F=A\cup_{L\sqcup U} B$, and denote the exteriors of $A$ and $B$ inside the 4-ball by $W_A$ and $W_B$ so that $X_F:= S^4 \setminus \nu(F) = W_A \cup_{X_{L\sqcup U}} W_B$.

For a given CW-complex $X$, let $b_i ^\omega(X) := \mathbf{dim}_{\mathbb{C}}H_i (X;\mathbb{C}_\omega)$ and   $\chi^\omega(X) := \sum_{i\ge 0} (-1)^i b_i ^\omega (X)$. We know that $\chi^\omega(X)$ coincides with the usual Euler characteristic $\chi(X):=\sum_{i\ge 0} (-1)^i \mathbf{dim}_{\mathbb{C}} H_i(X;\mathbb{C})$ provided $X$ is of finite type. In particular, $\chi^\omega (X_F) = \chi(X_F)$. 

By the  Mayer-Vietoris sequence, we know $\chi(X_F)=\chi(F\times S^1)-\chi(\nu(F))+\chi(S^4)$. Since $\chi(F\times S^1 )=0$, we get 
\[
\chi^\omega(X_F)= 2-\chi(F)=2-(4-2g_{st}(L))=2g_{st}(L)-2.
\]
Note also that $\chi^\omega(X_{L \sqcup U}) = \chi(X_{L \sqcup U})=0$.

By the Mayer-Vietoris sequence again, we have
\begin{align*}
\chi^ \omega(X_F) =2g_{st}(L)-2 & = \chi^ \omega (W_A) + \chi^ \omega(W_B) -\chi^\omega (X_{L\sqcup U})\\
&=\sum_{i=1}^3 (-1)^i (b^\omega _i (W_A)  +b^\omega_ i (W_B)).
\end{align*}
Due to \cite[Proposition 3.3]{conway2021}, we know that 
\[
|\sigma_{L\sqcup U}(\omega)| \le b_2 ^\omega(W_{\ast})+b^\omega _1 (W_{\ast}) - b_3 ^\omega (W_{\ast})-b_1^\omega(X_{L\sqcup U})
\]
where $*$ is either $A$ or $B$. This yields
\[
2g_{st}(L)-2 \ge 2|\sigma_{L\sqcup U}(\omega)| -2 b_1 ^\omega (W_A) -2 b_1 ^\omega(W_B) +2b_1 ^\omega(X_{L\sqcup U}). 
\]
From a part of the Mayer-Vietoris sequence,
\[
0\to\mathbf{ker}(\psi)\to H_1(X_{L\sqcup U};\mathbb{C}_\omega) \overset{\psi}{\to} H_1(W_A;\mathbb{C}_\omega) \oplus H_1 (W_B;\mathbb{C}_\omega) \to H_1(X_F;\mathbb{C}_\omega) \to 0,
\]
we observe that $b_1 ^\omega(X_F)=b_1 ^\omega(W_A)+b_1 ^\omega (W_B) -b_1 ^\omega(X_{L\sqcup U})+\mathbf{dim}_{\mathbb{C}}\mathbf{ker}(\psi)$. By Lemma \ref{lem1}, we know that $b_1 ^\omega(X_F)\le 1$. Thus, 
\[
 b_1 ^\omega(W_A)+b_1 ^\omega (W_B) -b_1 ^\omega(X_{L\sqcup U})\le 1.
\]
Therefore, 
\[
g_{st}(L) \ge |\sigma_{L\sqcup U}(\omega)|.
\]

It remains to compute $\sigma_{L\sqcup U}(\omega)$. Split union formula gives $\sigma_{L\sqcup U}(\omega)=\sigma_L (\omega)+\sigma_{U}(\omega)$. Also, since any 2-colored unlinked is (colored) doubly slice, we know from \cite{conway2021} that $\sigma_{U}(\omega)=0$. Therefore we get $\vert\sigma_L (\omega)\vert \le g_{st}(L)$.
\end{proof}

\section{Proof of the main theorem}

\subsection{Proof of Theorem \ref{mainthm1}}
This subsection is devoted to the proof of Theorem \ref{mainthm1}. Our key players are Bing doubles and its branched double cover along an unknot component. 

We give a brief summary of the main ingredients of our proof. First of all, on one hand, we prove that $g_{st}(B(K))\le 1$ for any slice knot $K$ where $B(K)$ denotes the Bing double of $K$. On the other hand, we will prove that $2|\sigma_K (\omega)|\le g_{ds}(B(K))$ for any slice knot $K$, where $\sigma_K(\omega)$ is the multivariable signature. This bound is useful because the multivariable signature is not a concordance invariant for some value of $\omega$. Indeed, we will use the slice knot $K=\sharp^n 8_{20}$ and $\omega=e^{\pi i/3}$ so that $\sigma_K(\omega)=n$. These establish Theorem \ref{mainthm1}. 

As promised, we first show that the stabilization genus of the Bing double of a slice knot is at most 1. This is a consequence of the following more general lemma. 
\begin{lem}
For any $2$-knot $F$, $g_{st}(B(F)) \le 1$.
\end{lem}

\begin{proof}
We can consider a banded unlink diagram of $B(F)$ as in the top left of Figure \ref{bing}. Then we can stabilize $B(F)$ and then destabilize it to obtain a $2$-component trivial link as in Figure \ref{bing}. 
\end{proof}

\begin{figure}[!hb]
    \centering
    \includegraphics[width=\textwidth]{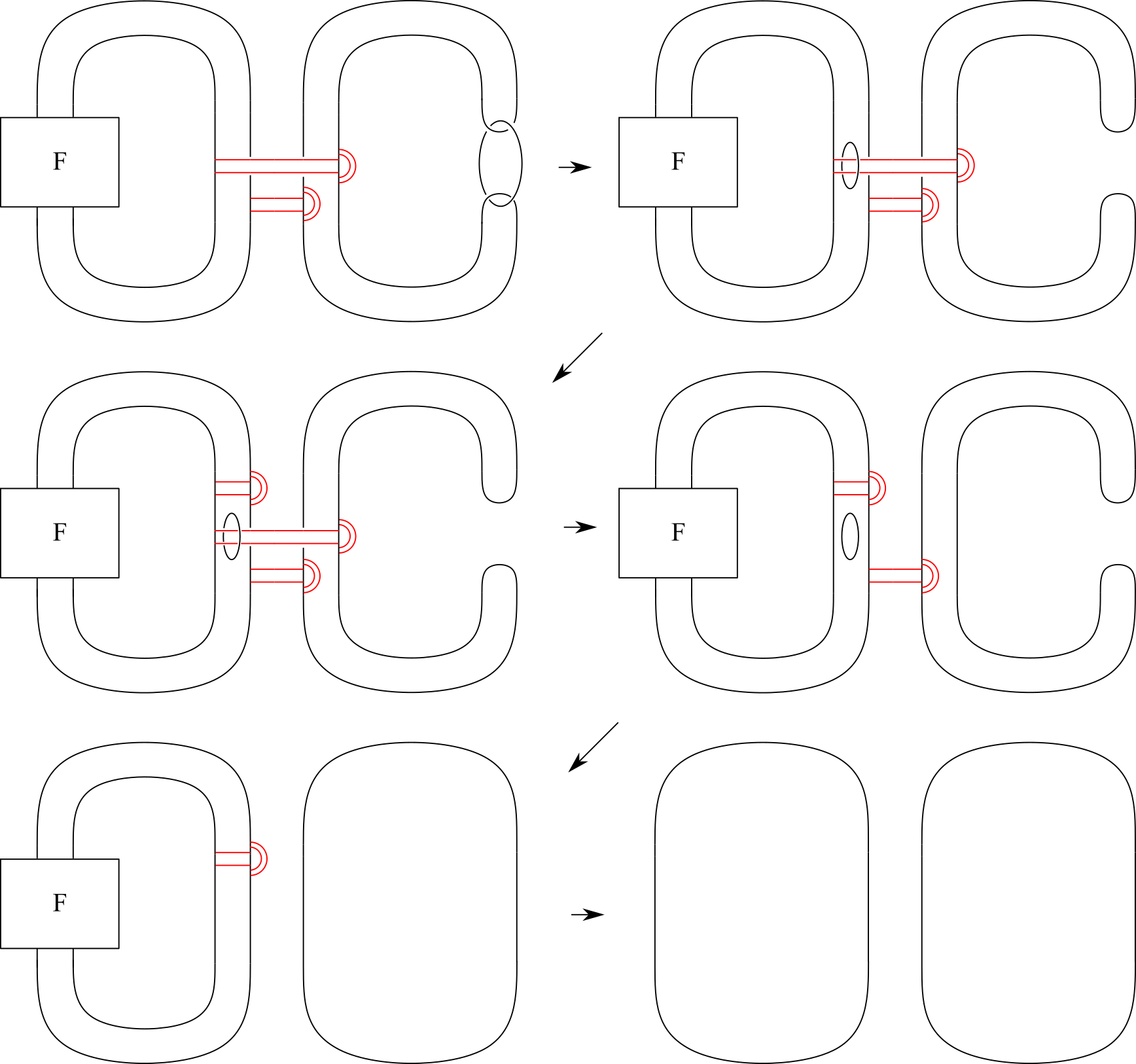}
    \caption{\textbf{Top left}: A diagram of a Bing double of $F$. Note that the left side of the figure is consists of two parallel copies of $F$. \textbf{Top right}: We can slide one component, which represented by an unknot, through the diagram so that it linked with a band as in the figure. \textbf{Middle left}: Stabilize one of the components. \textbf{Middle right}: Destabilize it in a different location. \textbf{Bottom left}: Two components are split. \textbf{Bottom right}: The left component of previous figure is an untwisted tubing of two parallel copies of $F$, so it bounds a 3-ball, which means it is unknotted. Hence we get a trivial $2$-component link. }
\label{bing}
\end{figure}

\begin{cor}
For any slice knot $K$, we have $g_{st}(B(K)) \le 1$.
\end{cor}

Now we give a lemma that will be used in our proof of Theorem \ref{mainthm1}. This lemma relates the multivariable signature of a link $L\subset S^3$ that bounds disks in a 4-manifold $W$ (not necessarily the 4-ball) and the signatures of the ambient manifold $W$.

Now we are ready to prove Theorem \ref{mainthm1}.

\begin{proof}[Proof of Theorem \ref{mainthm1}]
Given a slice knot $K$, suppose that $B(K)=L_1 \cup L_2$ is the cross-section of a trivial link $F=F_1 \cup F_2$, with genera $g(F_1)=g_1$ and $g(F_2)=g_2$. Consider the branched cover of $S^4$ along $F_1$. Since $F_1$ is an unknotted surface of genus $g_1$, the branched cover is $\sharp ^{g_1} (S^2 \times S^2)$. Since $L_1$ is an unknot in $S^3$, the preimage of the equatorial $S^3$ on the branched double cover is again $S^3$, which splits it as a connected sum of two 4-manifolds, say $A$ and $B$. Since we already know that $B(K)$ is a 2-component boundary link with unknotted (thus doubly slice) components, which is weakly doubly slice for both of its quasi-orientations, it remains to choose $K$ so that $g_1 +g_2$ is arbitrarily large for any choice of $F$.

The preimage of the component $L_2$ of $B(K)$ is a 2-component link; choose one component $T$. The slice-surfaces of $L_2$ which glue together to form $F_2$ lift to slice-surfaces $F_A$ (in $A$) and $F_B$ (in $B$) of $T$. Since $F_2$ bounds a handlebody of genus $g_2$ which is disjoint from a handlebody of genus $g_1$ bounded by $F_1$, we see that $F_T =F_A \cup_T F_B$ also bounds a handlebody of genus $g_2$ in $\sharp ^{g_1} (S^2 \times S^2)$. This implies that $F_T$ is null-homologous; since the homology class of $F_T$ is represented as the image of the pair of homology classes of $F_A$ and $F_B$ under the gluing isomorphism 
\[
H_2 (A;\mathbb{Z}) \oplus H_2 (B;\mathbb{Z}) \xrightarrow{\simeq} H_2 (\sharp ^{g_1} (S^2 \times S^2);\mathbb{Z}),
\]
the surfaces $F_A$ and $F_B$ are also null-homologous. Furthermore, Since $A\sharp B\simeq \sharp ^{g_1} (S^2 \times S^2)$, we have $H_1 (A;\mathbb{Z})\simeq H_1 (B;\mathbb{Z})\simeq 0$. Hence, we apply Theorem 3.5 of Conway-Nagel \cite{conway2020} to get
\[
\begin{split}
    \mathbf{sign}_{\omega} (W_{F_A}) &= \sigma_T(\omega)+\mathbf{sign}(A), \\
    \mathbf{sign}_{\omega} (W_{F_B}) &= \sigma_T(\omega)+\mathbf{sign}(B),
\end{split}
\]
where $W_{F_A}$ and $W_{F_B}$ are the exteriors of $F_A$ and $F_B$, respectively.

Since $\mathbf{sign}(\sharp^{g_1} (S^2\times S^2)) = 0$, and since $H_2 (A;\mathbb{Z}) \oplus H_2 (B;\mathbb{Z}) \simeq H_2 (\sharp ^{g_1} (S^2 \times S^2);\mathbb{Z})$, we have $\mathbf{sign}(A)+\mathbf{sign}(B)=0$. This yields
\[
2\sigma_T(\omega) = \mathbf{sign}_\omega(W_{F_A}) + \mathbf{sign}_\omega(W_{F_B}).
\]

Now we follow the proof of \cite[Theorem 3.4]{conway2021}. Since $\pi_1((\sharp ^{g_1} (S^2 \times S^2)) - F_T) \simeq \mathbb{Z}$ and $\chi((\sharp ^{g_1} (S^2 \times S^2)) - F_T) = 2(g_1 +g_2)-2$, we deduce that 
\[
2\vert \sigma_T (\omega)\vert \le \vert \mathbf{sign}_{\omega}(W_{F_A}) \vert + \vert \mathbf{sign}_{\omega}(W_{F_B}) \vert \le 2(g_1 +g_2).
\]

As shown in Figure \ref{bingcover}, $T$ is isotopic to $K\sharp K^r$, where $K^r$ denotes the knot $K$ with orientation reversed. Since Levine-Tristram signature is invariant under orientation reversal, we have $\sigma_T (\omega) = 2\sigma_K (\omega)$. Thus, by taking $F$ to be the trivial surface-link which realizes $g_{ds}(B(K))$, we get the inequality 
\[
2\vert \sigma_{K}(\omega) \vert \le g_{ds}(B(K)).
\]

Finally, given any positive integer $n$, set $K_n$ to be knot obtained by performing a connected sum of $n$ copies of a slice knot $8_{20}$, and take $L_n =B(K_n)$. Then for $\omega = e^{\frac{\pi i}{3}}$, we have $\sigma_{K}(\omega) = n\sigma_{8_{20}}(\omega) = n$, so we get $g_{ds}(L_n) \ge 2n$ as desired.

\begin{figure}[!hb]
    \centering
    \includegraphics[width=\textwidth]{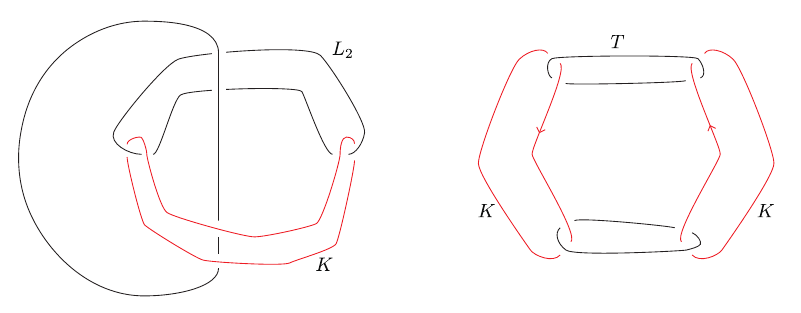}
    \caption{\textbf{Left}: The Bing double $B(K)$. \textbf{Right}: The component $T$ in the preimage of $L_2$}
    \label{bingcover}
\end{figure}
\end{proof}

\begin{rem}
While we used Bing doubles to prove Theorem \ref{mainthm1}, we can also use more complicated patterns induced by other Brunnian links to construct more examples satisfying similar properties. To see how, recall the pattern $P_{L_n,G}$ in Figure \ref{newbrunnian}, and write $P_{L_n,G}=R\cup B$, where $R$ and $B$ are the components of $L_n$ drawn in red and blue, respectively. Let $K_m$ be the knots used in the proof of Theorem \ref{mainthm1} and consider the link $P_{L_n,G}(K_m)$, where we again denote its components by $R$ and $B$. After taking branched double cover over $B$, the lifts of $R$ are isotopic to $(K_m)_{n,1}\sharp (K_m)^{r}_{n,-1}$. Then, using the arguments used in the proof of \ref{mainthm1}, we deduce that $g_{ds}(P_{L_n,G}(K_m))\ge 2m$ and $g_{st}(P_{L_n,G}(K_m))=1$.
\end{rem}

\subsection{Proof of Theorem \ref{mainthm4}} Now we consider higher dimensional cases.  Here, we recall the Bing double construction for $n$-links.
\begin{defn}
Consider a 2-component trivial link $L\subset S^{n+2}$. Let $\gamma \subset S^{n+2} - L$ be a simple closed curve such that $[\gamma] = xyx^{-1}y^{-1} \in \pi_1 (S^{n+2} - L)$ where $x$ and $y$ are meridians of each component of $L$ which generate $\pi_1 (S^{n+2} - L)$. The Bing double of a spherical knot $F$ is the 2-component spherical link obtained by a satellite construction (see Definition \ref{satellitedefn}), using $F$ as a companion and $(L,\gamma)$ as a pattern.  See Figure \ref{bingcover} for a schematic figure. 
\end{defn}

The following lemma is an immediate modification of Theorem \ref{satthm}. 

\begin{lem}
For a slice $n$-knot $K$, its Bing double $B(K)$ is always boundary and weakly doubly slice.
\end{lem}

Stoltzfus \cite{stoltzfus1978} defined the \emph{doubly null concordance group} $\mathcal{CH}_{n}$ of $n$-knots. Briefly speaking, this group consists of the equivalent classes of $n$-knots where $K$ and $L$ are equivalent if and only if there are doubly slice $n$-knots $M$ and $N$ such that $K\sharp M$ is isotopic to $L\sharp N$.  When $n=2q-1$ is odd, we consider the group 
\[
\mathcal{CH}_{2q-1} ^ {q-1}=\{[K]\in \mathcal{CH}_n\,|\,K \text{ has a }(q-1)\text{-connected Seifert solid}\}.
\]
Recall an $n$-knot is called \emph{simple} if it admits a $[(n-1)/2]$-connected Seifert solid. With this terminology,  $\mathcal{CH}_{2q-1} ^ {q-1}$ is the the group of stable concordance classes of simple $(2q-1)$-knots. 

An algebraic counterpart of $\mathcal{CH}_{2q-1} ^{q-1}$ is so-called \emph{double Witt group} $CH^{(-1)^q} (\mathbb{Z})$. Here we give a brief definition of $CH^{\varepsilon}(\mathbb{Z})$, $\varepsilon = \pm 1$. We first define (algebraic) Seifert forms. 
\begin{defn}
Let  $\varepsilon=\pm1$.  An $\varepsilon$-symmetric Seifert form (or Seifert isometric structure, depending on literature) over $\mathbb{Z}$ is a triple $(N,b,t)$ where 
\begin{itemize}
    \item $N$ is a finitely generated projective $\mathbb{Z}$-module,
    \item $b:N\to N^*$ is an isomorphism of $\mathbb{Z}$-modules which satisfies $b(x)(y) = \varepsilon b(y)(x)$ for all $x,y\in N$,
    \item $t:N\to N$ is an endomorphism such that $b(tx)(y) = b(x)((1-t)y)$.
\end{itemize}
\end{defn}
A Serfert form $(N,b,t)$ is hyperbolic if there are  $t$-invariant $\mathbb{Z}$-submodules $L_1,L_2\subset N$ such that $N=L_1\oplus L_2$ and $L_i^\perp:=\{y\in N\,|\, b(y,x)=0\text{ for all }x\in L_i\}=L_i$. We say that two Seifert forms $K$ and $L$ are equivalent if and only if there are hyperbolic Seifert forms $M,N$ such that $K\oplus M$ is isometric to $L\oplus N$ where $\oplus$ is the orthogonal sum. The group $CH^\varepsilon (\mathbb{Z})$ consists of the equivalent classes of $\varepsilon$-symmetric Seifert forms. It is known that, for $q>1$, there is an isomorphism $\mathcal{CH}^{q-1}_{2q-1}\to CH^{(-1)^{q}}(\mathbb{Z})$ 
\cite{stoltzfus1978,levine1977,sumners1971}. This isomorphism maps a knot $K$ with a Seifert solid $V$ to the equivalent class of the Seifert form $(fH_q(V;\mathbb{Z}), b, t)$ where
\begin{itemize}
\item  $f H_q(V;\mathbb{Z})$ is the free part of $H_q(V;\mathbb{Z})$,
\item $b$ is the intersection product,
\item $t$ is given by the  property that $(i^+ _* - i^- _*) (t x) = i^+_* (x)$, $x\in fH_q(V;\mathbb{Z})$. Here $i^{\pm}_* : fH_q(V;\mathbb{Z}) \to fH_q(S^{2q+1}\setminus K;\mathbb{Z})$ are induced maps of the positive/negative push-off. 
\end{itemize}

A Seifert matrix associated to the choice of a Seifert solid $V$ of $K$ is the matrix representation of the pairing $fH_q(V;\mathbb{Z}) \otimes fH_q(V;\mathbb{Z}) \to \mathbb{Z}$ given by $x\otimes y \mapsto b(tx,y)$. Conversely, given a $(2g\times 2g)$ integral matrix $\psi$ such that $\psi+\epsilon \psi^T$ is unimodular, we can find a Seifert form whose Seifert matrix is $\psi$ by letting $N=\mathbb{Z}^{2g}$, $b=\psi+\varepsilon \psi^T$, and $t=b^{-1}\psi$. Observe that $\psi$ is hyperbolic if and only if its Seifert form is hyperbolic. 

We recall the following general obstruction for a matrix being hyperbolic.
\begin{lem}\label{hyperbolic}
Let $A$ be a $(2g\times 2g)$-matrix over $\mathbb{R}$. If $A$ is hyperbolic then the signature of the Hermitian matrix  $(1-\omega)A+(1-\overline{\omega})A^T$ vanishes for all $\omega\in S^1\setminus \{1\} \subset \mathbb{C}$.
\end{lem}
\begin{proof}
If $A$ is hyperbolic, we can find an invertible matrix $C$ such that $C^T A C = \begin{pmatrix} 0 & P \\ Q & 0\end{pmatrix}$ for some $(g\times g)$-matrices $P$ and $Q$. Hence, $(1-\omega)A + (1-\overline{\omega})A^T$ is congruent to the matrix 
\[
A_\omega:=\begin{pmatrix} 0 & X_\omega \\ (X_\omega )^* & 0 \end{pmatrix},
\]
where $X_\omega=(1-\omega) P + (1-\overline{\omega}) Q^T$.

We need to show that the signature of $A_\omega$ is always 0 regardless of $\omega$. In fact, this follows from the observation that the characteristic polynomial of $A_\omega$ is of the form $p_\omega (x^2)$ where $p_\omega (x)$ is the characteristic polynomial of $(X_\omega )^*X_\omega$. 
\end{proof}

We consider the slice knot $K=8_{20}$ which appeared in the proof of Theorem \ref{mainthm1}. This knot has the Seifert matrix
\[
\psi=\begin{pmatrix}
 -1 &  -1 & -1 & -1\\   0 & 0 & -1 & -1\\   0& -1 & 0& -1\\  0& 0& -1& 0
\end{pmatrix}.
\]
The associated Seifert form $A=(N,b,t)$ of $K$ is then given by
\begin{itemize}
    \item $N=\mathbb{Z}^4$,
    \item $b$ is given in the matrix form
    \[
    b=\begin{pmatrix} 0 & -1 & -1 & -1 \\ 1& 0 & 0 & -1 \\ 1 & 0& 0& 0 \\ 1& 1& 0& 0 \end{pmatrix},
    \]
    \item $t$ is the endomorphism $N\to N$ defined in terms of the matrix
    \[
    t=\begin{pmatrix} 0 & -1 & 0 & -1 \\ 0& 1 & -1 & 1 \\ 1 & 1& 1& 0 \\ 0& -1& 1& 0 \end{pmatrix}.
    \]
\end{itemize}
We already know that $K\sharp K^r$ is not doubly slice as its signature function is not constantly vanishing. According to Lemma \ref{hyperbolic}, this also implies that the Seifert matrix $\begin{pmatrix} \psi & 0 \\ 0 & \psi^T\end{pmatrix}$ of $K\sharp K^r$ is metabolic but not hyperbolic. Now for each odd $q\ge 1$, we can find the simple slice $(2q-1)$-knot $F$  whose  Seifert form is $A$ \cite{levine1969}. Then we know that $F\sharp F^r$ is not doubly slice \cite{sumners1971}.

When $q$ is even and $n=2q-1$ is odd, we consider the Seifert form $A=(N,b,t)$ given by 
\begin{itemize}
    \item $N=\mathbb{Z}^4$, a rank 4 free abelian group,
    \item $b$ is given in the matrix form
    \[
    b=\begin{pmatrix} 0 & 0 & 1 & -1 \\ 0& 0 & 1 & 0 \\ 1 & 1& 2& 0 \\ -1& 0& 0& 2 \end{pmatrix},
    \]
    \item $t$ is the endomorphism $N\to N$ defined in terms of the matrix
    \[
    t=\begin{pmatrix} 0 & -1 & 2 & -1 \\ 1& 1 & -3 & 3 \\ 0 & 0& 1& -1 \\ 0& 0& 1& 0 \end{pmatrix}.
    \]
\end{itemize}
This Seifert form is induced from the Seifert matrix 
\[
\psi=\begin{pmatrix} 0 & 0& 0& -1\\ 0&0&1&-1 \\ 1&0&1&0 \\ 0&1 &0&1\end{pmatrix}
\]
in the way that $b=\psi+\psi^T$, and $t=b^{-1}\psi$ as we mentioned above. We will show that $\psi$ is metabolic but not hyperbolic. This implies that the Seifert form $A$ is metabolic but not hyperbolic.

In fact, it suffices to show that $\psi$ is not hyperbolic since $\psi$ is clearly metabolic. For this, we compute the ``signature function'' again. By direct computation, we see that, when $\omega=e^{2\pi i /3}$, the matrix $(1-\omega) \psi + (1-\overline{\omega}) \psi^T$ has eigenvalues $0, 3, \frac{3-\sqrt{57}}{2}$ and $\frac{3+\sqrt{57}}{2}$. Hence, its signature is 1. Then Lemma \ref{hyperbolic} shows that $\psi$ cannot be hyperbolic.

Therefore, we deduce that $A$ is metabolic and $A\oplus A$ is not hyperbolic. By \cite{levine1969}, we can find a simple slice $n$-knot $F$ such that $F\sharp F^r$ is not doubly slice. 

Now we consider the case when $n=2q$ is an even integer. In this case, Stoltzfus \cite{stoltzfus1978} showed that there is an surjective homomorphism $\mathcal{CH}_{2q} \to CH^{(-1)^q}(\mathbb{Q}/\mathbb{Z})$ provided $q>1$.
On the other hand, the group $CH^{(-1)^q}(\mathbb{Q}/\mathbb{Z})$ consists of infinite direct sums of $\mathbb{Z}_2$ and $\mathbb{Z}_4$ \cite{hillman1986}. In particular, there is an $n$-knot $F$ such that $F\sharp F$ is not doubly slice. Note that $F$ is slice since every even dimensional knot is slice \cite{kervaire1965}. 

The above argument proves the following lemma:

\begin{lem}\label{notdoublyslice}
Let $n$ be any natural number not equal to 2. There is a slice $n$-knot $K$ such that $K\sharp K^r$ is not doubly slice. 
\end{lem}

Now by adopting the main idea of the proof of Theorem \ref{mainthm1}, we can obtain Theorem \ref{mainthm4}.

\begin{proof}[Proof of Theorem \ref{mainthm4}]
Take an $n$-knot $K$ as in Lemma \ref{notdoublyslice}. Let $F:=B(K)$ be the Bing double of $K$. $F$ is a unknotted, weakly doubly slice, boundary link. Hence it suffices to show that $F$ is not strongly doubly slice. Let $F_1$, $F_2$ be components of $F$.

On the contrary, suppose that $F$ is strongly doubly slice. Let $D=D_1\cup D_2$ be a trivial 2-component $(n+1)$-link in $S^{n+3}$ such that $D_i \cap S^{n+2} = F_i$. We take a branched double cover $\Sigma$ of $S^{n+3}$ along $D_1$. As $D_1$ is trivial, we know that $\Sigma$ is diffeomorphic to $S^{n+3}$. Note that the lift of $F_2$  is isotopic to $K \sharp K^r$. Observe that $K \sharp K^r$ is doubly slice since it is a cross-section of the lift of $D_2$. This contradicts our choice of $K$. 
\end{proof}

\begin{rem}
Our approach to prove Theorem \ref{mainthm4} (and Lemma \ref{notdoublyslice}) cannot cover the $n=2$ case. The main difficulty is that the realization problem is not established yet. Namely, we do not know whether each element of $CH^\varepsilon (\mathbb{Q}/\mathbb{Z})$ can be realized as a Blanchfield form of some $2$-knot. 
\end{rem}

\subsection{Proof of Theorem \ref{mainthm2}}\label{Topologicallynotsmoothly}

We now move on to prove Theorem \ref{mainthm2}. Recall that, when $L$ is a Brunnian link with a distinguished component $A$ and $K$ is doubly slice, then $P_{L,A}(K)$ is strongly doubly slice. As mentioned in the introduction, this statement holds in both locally flat and smooth categories. Thus, if we start with a companion $K$ which is topologically but not smoothly doubly slice, then $P_{L,A}(K)$ is a topologically strongly doubly slice 2-component link which might not be smoothly strongly slice. This observation leads us to the following proof.

\begin{proof}[Proof of Theorem \ref{mainthm2}]
Let $K$ be any of the knots which satisfies \cite[Theorem B]{meier2015}. In particular, $K$ is smoothly slice and topologically doubly slice, but $\sharp^n K$ is not smoothly doubly slice for any positive integer $n$; note that $K\sharp K^r$ is also not smoothly doubly slice, since Meier's arguments are insensitive under orientation reversal. Then its Bing double $B(K)$ is a 2-component boundary link consisting of unknotted(thus doubly slice) components, which is topologically strongly doubly slice and smoothly weakly doubly slice for both of its quasi-orientations. It only remains to prove that $B(K)$ is not smoothly strongly doubly slice.

Suppose in contrary that $B(K)$ is smoothly strongly doubly slice. Then it is a smooth cross-section of the trivial 2-component spherical link $F$. As in the proof of Theorem \ref{mainthm1}, we take double branched cover of everything over along a component of $F$ and choose a component $T\simeq K\sharp K^r$ of the lift of the other component of $B(K)$. Then we get a splitting $S^4 = U\cup_{S^3} V$ for some smooth 4-manifolds $U$ and $V$ with $\partial U =\partial V=S^3$, and $T$ admits slice disks $D_U \subset U$ and $D_V \subset V$ such that $D_U \cup D_V$ is a smooth unknotted sphere. 

The topological Schoenflies theorem \cite{brown1960proof, mazur1961embeddings} ensures that both $U$ and $V$ are topologically a $4$-ball, so the double branched covers $\Sigma(D_U)$ and $\Sigma(D_V)$ of $U$ and $V$ along the disks $D_U$ and $D_V$ give metabolizers of $H_1 (\Sigma(T);\mathbb{Z})$, where $\Sigma(T)$ is the branched double cover of $S^3$ along $T$. Furthermore, $\Sigma(D_U)$ and $\Sigma(D_V)$ are rational homology $4$-balls. Since Meier's arguments rely only on the existence of metabolizers in the branched double cover and the fact that $d$-invariants are $\mathbb{Q}$-homology $\mathbf{Spin}^c$-cobordism invariants, we see that $D_U \cup D_V$ cannot be a smooth unknotted sphere, a contradiction. Therefore $B(K)$ is not smoothly strongly doubly slice.
\end{proof}


\bibliographystyle{amsalpha}
\bibliography{ref}
\end{document}